\documentclass[10pt,final,twocolumn]{IEEEtran}
\long\def\comment#1{}

\usepackage{textcomp}
\usepackage{cite}
\usepackage{graphicx}
\usepackage{subfigure}
\usepackage{psfrag}
\usepackage{url}
\usepackage{stfloats}
\usepackage{amsmath}
\usepackage{amssymb,amsfonts}
\usepackage{amsthm}
\usepackage{float}
\usepackage[usenames]{color}
\usepackage{algorithm,algorithmic}

\newtheorem{corollary}{Corollary}
\newtheorem{assumption}{Assumption}
\newtheorem{remark}{Remark}

\newtheorem{lemma}{Lemma}
\newtheorem{theorem}{Theorem}
\newtheorem{example}{Example}

\begin{document}

\setlength{\arraycolsep}{0.3em}

\title{Distributed Online Convex Optimization with an Aggregative Variable
\thanks{}}

\author{Xiuxian Li, Xinlei Yi, and Lihua Xie
\thanks{Corresponding author: L. Xie.}
\thanks{X. Li and L. Xie are with School of Electrical and Electronic Engineering, Nanyang Technological University, 50 Nanyang Avenue, Singapore 639798 (e-mail: xxli@ieee.org; elhxie@ntu.edu.sg).}
\thanks{X. Yi is with the Division of Decision and Control Systems, School of Electrical Engineering and Computer Science, KTH Royal Institute of Technology, 100 44, Stockholm, Sweden (e-mail: xinleiy@kth.se).}
}

\maketitle

\setcounter{equation}{0}
\setcounter{figure}{0}
\setcounter{table}{0}

\begin{abstract}
This paper investigates distributed online convex optimization in the presence of an aggregative variable without any global/central coordinators over a multi-agent network, where each individual agent is only able to access partial information of time-varying global loss functions, thus requiring local information exchanges between neighboring agents. Motivated by many applications in reality, the considered local loss functions depend not only on their own decision variables, but also on an aggregative variable, such as the average of all decision variables. To handle this problem, an Online Distributed Gradient Tracking algorithm (O-DGT) is proposed with exact gradient information and it is shown that the dynamic regret is upper bounded by three terms: a sublinear term, a path variation term, and a gradient variation term. Meanwhile, the O-DGT algorithm is also analyzed with stochastic/noisy gradients, showing that the expected dynamic regret has the same upper bound as the exact gradient case. To our best knowledge, this paper is the first to study online convex optimization in the presence of an aggregative variable, which enjoys new characteristics in comparison with the conventional scenario without the aggregative variable. Finally, a numerical experiment is provided to corroborate the obtained theoretical results.
\end{abstract}


\section{Introduction}\label{s1}

Online optimization/learning is a sequence of decision making processes, where a sequence of time-varying (and possibly adversarial) loss functions are revealed gradually to the decision maker. Online optimization has numerous practical applications such as in machine learning, auctions, dictionary learning, portfolio management, and neural networks \cite{zinkevich2003online,hazan2007logarithmic,mairal2009online,shalev2012online}, to name just a few.

With the development of advanced communication and computing technologies and the emergence of large-scale datasets, distributed online optimization has become a hot topic in recent two decades, where a finite group of agents, such as robots, computing units, processors, autonomous vehicles, and sensors, aim to solve a global online optimization problem in a cooperative manner by local information exchanges between neighboring agents. It should be noted that each individual agent can access only partial information on the global problem, and the partial information may be private to each agent who is unwilling to expose the information to others. Usually, there exists a global (central) coordinator/master in centralized online optimization, while it is not practical in distributed online optimization. In contrast to centralized online optimization, the distributed case has overwhelming advantages, including lower cost, better robustness, and privacy preservation, etc.

The studied problem in distributed online optimization is generally in the form of $\sum_{i=1}^N f_{i,t}(x_i)$ subject to $x_i=x_j$ for all $i,j\in[N]$ and possible equality/inequality constraints, where $N$ is the number of agents in the network, from which one can observe that each local loss function $f_{i,t}$ depends only on its own decision variable $x_i$. However, in many realistic applications, such as warehouse location problem, transportation systems, signal processing, target surrounding by robots and unmanned aerial vehicles (UAVs), the local loss functions usually rely on other agents' decision variables besides its own variable. For example, in target surrounding, a collection of agents (such as robots, UAVs, or autonomous vehicles) desire to surround a target in order to protect the target from the attack of intruders, and in this case, local loss functions will rest not only on its own decision variable (such as position), but also on the average of the decision variables of all agents. Theoretically, no existing algorithms are available to the case with an aggregative variable, thus making it a challenging problem.

Motivated by the above facts, this paper is concerned with the scenario where each local loss function depends not only on its own variable, but also on an aggregative variable, which is a global variable and not known to any individual agent.

\subsection{Related Works}

Online convex optimization (OCO) was firstly studied in a centralized setup, including the case with only feasible set constraints \cite{zinkevich2003online,shalev2012online,chen2018projection} (having the optimal static regret bound $O(\sqrt{T})$), the case with static inequality constraints \cite{mahdavi2012trading,jenatton2016adaptive,yuan2018online}, and the case with time-varying inequality constraints \cite{yu2017online,chen2017online,sun2017safety,cao2018online,wei2019online}.

As for distributed OCO, which is our main focus in this paper, various scenarios have been addressed in the literature, such as \cite{zhang2017projection,mateos2014distributed,akbari2017distributed,nedic2015decentralized,koppel2015saddle,shahrampour2018distributed,hosseini2016online,lee2017stochastic,yuan2017adaptive,lee2017sublinear}, to quote a few. For instance, distributed online unconstrained optimization problems have been investigated in \cite{mateos2014distributed} and \cite{akbari2017distributed}, where an online subgradient descent algorithm and a distributed online subgradient push-sum algorithm are proposed, respectively. Meanwhile, many algorithms have been developed for the case with global/local set constraints in the literature, such as, a variant of the Arrow-Hurwicz saddle point algorithm \cite{koppel2015saddle}, Nesterov based primal-dual algorithm \cite{nedic2015decentralized}, dual subgradient averaging algorithm \cite{hosseini2016online}, distributed primal-dual algorithm \cite{lee2017stochastic}, and mirror descent algorithm \cite{shahrampour2018distributed}. Moreover, local static inequality constraints have been addressed in \cite{yuan2017adaptive} by a consensus-based adaptive primal-dual subgradient algorithm. Smart grid networks have been discussed as an application of distributed online optimization in \cite{zhou2017incentive}. Furthermore, a general constraint, i.e., a static coupled inequality constraint, has been considered in \cite{lee2017sublinear} and \cite{li2018distributedon}, where a sublinear static regret is ensured by distributed primal-dual algorithms. More recently, time-varying coupled inequality constraints have been studied in \cite{yi2019distributed} and \cite{yi2019distributed2} with full gradients and bandit feedback, respectively.

It can be found that all aforementioned works focus on the case where each local loss function is dependent only on its own decision variable. Inspired by the practical applications, this paper investigates the case where an aggregative variable, building upon all agents' decision variables, is involved in each local loss function, thus requiring new techniques for dealing with this new problem.

\subsection{Our Contributions}

To our best knowledge, this paper is the first to investigate online optimization with an aggregative variable. To tackle this problem, a novel algorithm with full gradients, called online distributed gradient tracking algorithm (O-DGT), is developed. It is rigorously proved that the dynamic regret is upper bounded by three terms: a sublinear term $O(\sqrt{T})$, a path variation term, and a gradient variation term, thus obtaining a sublinear dynamic regret when the path and gradient variation terms are both sublinear. Note that even in the conventional case without the aggregative variable, it is necessary for the dynamic regret to be bounded by the path variation, since achieving a sublinear dynamic regret is impossible if the path variation is too large \cite{yi2019distributed}. Also, the gradient variation is a new term required for the dynamic regret bound, which is generally unnecessary for the case without the aggregative variable. The reason behind this phenomenon is that the current online optimization with the aggregate variable also needs to estimate the gradients of other agents' loss functions. Moreover, when using a constant stepsize, the upper bound on dynamic regret for the distributed online optimization can be established which is almost the same as the centralized algorithms, except that it requires to know the information on $T$, the path variation, and the squared gradient variation beforehand, which are also used in centralized algorithms.

On the other hand, instead of true gradient information, we further study the O-DGT algorithm with stochastic/noisy gradients. It is shown rigorously that the expected dynamic regret has the same upper bound as the full gradient case.

As by-products, the aforementioned two results can be applied to the case when all loss functions are static, that is, all functions are independent of time. In this case, the O-DGT algorithm is renamed a distributed gradient tracking algorithm (DGT). For DGT with true gradients, it is shown that the algorithm is convergent to an optimizer at the rate of $O(1/\sqrt{T})$. Meanwhile, as for DGT with stochastic/noisy gradients, the algorithm can also be proved to be convergent to an optimizer in the sense of expectation with a rate $O(1/\sqrt{T})$. In comparison, the static case here is studied in the general convex setting, while \cite{li2020distributed} only addressed the strongly convex case.

{\em Notations:} Let $\mathbb{R}^n$ be the set of vectors with dimension $n>0$. Define $[k]=\{1,2,\ldots,k\}$ for an integer $k>0$. Denote by $col(z_1,\ldots,z_k)$ the column vector formed by stacking up $z_1,\ldots,z_k$. Let $\|\cdot\|$, $x^\top$, and $\langle x,y\rangle$ be the standard Euclidean norm, the transpose of $x\in\mathbb{R}^n$, and standard inner product of $x,y\in\mathbb{R}^n$, respectively. Let $\mathbf{1}$ and $\mathbf{0}$ be column vectors of compatible dimension with all entries being $1$ and $0$, respectively, and $I$ be the compatible identity matrix. $\otimes$ is the Kronecker product. Let $\nabla f$ and $Id$ denote the gradient of a function $f$ and the identity map, respectively.

\section{Preliminaries}\label{s2}

The {\em projection} of a point $x\in\mathbb{R}^n$ onto a closed convex set $S\subseteq\mathbb{R}^n$ is defined by $P_{S}(x):=\arg\min_{y\in S}\|x-y\|$, satisfying:
\begin{align}
\|P_S(x)-P_S(y))\|\leq \|x-y\|,~~~\forall x,y\in \mathbb{R}^n.       \label{2}
\end{align}

\subsection{Problem Formulation}\label{s2.3}

This paper considers a sequence of decision making problems, where there exist a sequence of time-varying (and maybe adversarial) loss functions $\{f_t\}_{t=0}^\infty$, which are called global loss functions and are separable. To be specific, $f_t$ consists of a sum of local loss functions $f_{i,t}$'s, i.e.,
\begin{align}
f_t(x)&=\sum_{i=1}^N f_{i,t}(x_i,\nu(x)),              \nonumber\\
\nu(x)&:=\frac{1}{N}\sum_{i=1}^N \psi_i(x_i)            \label{3}
\end{align}
for $x=col(x_1,\ldots,x_N)$, where $x_i\in X_i\subseteq\mathbb{R}^{n_i}$, $N$ is the number of agents involved in the problem, $\psi_i: X_i\to \mathbb{R}^d$ is a differentiable function for $i\in[N]$, and $\nu:\mathbb{R}^n\to\mathbb{R}^d$ is called an {\em aggregative variable} with $n:=\sum_{i=1}^N n_i$, since $\nu(x)$ represents an aggregative information of all decision variables, including the average $\sum_{i=1}^N x_i/N$ as a special case.

In this problem, $f_{i,t}$'s are revealed gradually, that is, for each $i\in[N]$, $f_{i,t}$ will be revealed to agent $i$ at time slot $t\geq 0$ only after agent $i$ has made its decision $x_{i,t}$. Also, each agent $i$ is only privately accessible to $f_{i,t}$ (along with its true/stochastic gradients) after making its decision $x_{i,t}$, without awareness of other local loss functions $f_{j,t}$'s for $j\neq i$. Moreover, $\psi_i$ is only privately known to agent $i$ for all $i\in[N]$, and each agent $i\in[N]$ only realizes its own decision variable $x_i$ without any knowledge of other agents' decision variables $x_j$'s for $j\neq i$.

The objective is to minimize the total loss over a time horizon $T>0$, i.e.,
\begin{align}
&\mathop{\min}_{x_1,\ldots,x_T\in X}~~~\sum_{t=1}^T f_t(x_t),          \label{1}\\
&\hspace{1.6cm}f_t(x_t):=\sum_{i=1}^N f_{i,t}(x_{i,t},\nu(x_t)),       \nonumber
\end{align}
where $X:=\prod_{i=1}^N X_i$ is the Cartesian product of $X_i$'s, and $x_{i,t}$ is the decision variable of agent $i$ made at time $t\geq 0$.

In doing so, a performance metric, called {\em dynamic regret}, is conventionally employed for online optimization, i.e.,
\begin{align}
\mathcal{R}_T:=\sum_{t=1}^T f_{t}(x_{t})-\sum_{t=1}^T f_{t}(x_{t}^*),        \label{4}
\end{align}
where $x_t:=col(x_{1,t},\ldots,x_{N,t})$, and $x_t^*:=\mathop{\arg\min}_{x\in X}f_t(x)$ is the best decision variable at time step $t$. Then, an algorithm is announced ``good'' if the dynamic regret $\mathcal{R}_T$ is sublinear with respect to $T$, i.e., $\mathcal{R}_T=o(T)$. It should be noted that many works have employed the static regret as a performance metric, which makes use of $x^*=\mathop{\arg\min}_{x\in X}\sum_{t=1}^T f_t(x)$ as a comparator, instead of $x_t^*$'s. Obviously, the dynamic regret is more meaningful as the global loss function is time-varying.

\begin{remark}\label{r1}
It is worth mentioning that the dependence on the aggregative variable for local loss functions has been studied in aggregative games \cite{liang2017distributed2,garg2017local}, which is however different from the scenario in this paper. The main difference lies in that all agents/players in aggregative games aim to minimize their own local loss/payoff functions in a noncooperative manner, while all agents in problem (\ref{1}) desire to minimize the sum of their local loss functions in a cooperative fashion. As a result, their solution sets and optimality conditions are distinct, as shown in a simple example below.
\end{remark}

\begin{example}\label{e1}
Consider a simple time-invariant case, i.e., $f_t=f$ for all $t\geq 0$, where $f$ is some function in the form $f(x)=\sum_{i=1}^N f_i(x_i,\nu(x))$. Let $N=2$, $n_i=1$, $\psi_1=\psi_2=Id$, $f_1=x_1^2+4\nu^2(x)$, and $f_2=(x_2-2)^2+4\nu^2(x)$, and in this case, $\nu(x)=(x_1+x_2)/2$. For distributed online optimization, the objective is to minimize $f=f_1+f_2$, and by $\nabla_{x_1} f=0$ and $\nabla_{x_2}f=0$, one can obtain the optimal decision $x_1=1.2$ and $x_2=-0.8$. On the other hand, for aggregative games, the aim is to minimize $f_i$ for agent $i$, respectively, and the optimality conditions are $\nabla_{x_1}f_1=0$ and $\nabla_{x_2}f_2=0$, which lead to the Nash equilibrium $x_1=-2/3$ and $x_2=4/3$. Apparently, the Nash equilibrium is different from the optimal decision $x_1=1.2$ and $x_2=-0.8$, since all agents in aggregative games are selfish who generally cannot collaborate as in distributed online optimization.
\end{example}

To move on, for notation simplicity, let $\nabla_1 f_{i,t}(x_i,\nu(x))$ and $\nabla_2 f_{i,t}(x_i,\nu(x))$ respectively denote $\nabla_{x_i}f_{i,t}(x_i,\nu(x))$ and $\nabla_{\nu}f_{i,t}(x_i,\nu(x))$ for all $i\in[N]$. And for $x\in\mathbb{R}^n$ and $y=col(y_1,\ldots,y_N)\in\mathbb{R}^{Nd}$, define $f_t(x,y):=\sum_{i=1}^N f_{i,t}(x_i,y_i)$, $\nabla_1 f_t(x,y):=col(\nabla_1 f_{1,t}(x_1,y_1),\ldots,\nabla_1 f_{N,t}(x_N,y_N))$ and $\nabla_2 f_t(x,y):=col(\nabla_2 f_{1,t}(x_1,y_1),\ldots,\nabla_2 f_{N,t}(x_N,y_N))$.

To proceed, it is necessary to postulate some standard conditions for the regret analysis.

\begin{assumption}\label{a1}
The following hold for problem (\ref{1}):
\begin{enumerate}
  \item $X_i$'s are nonempty, convex and compact, that is, there exists a constant $B>0$ such that $\|x_i\|\leq B$ for all $x_i\in X_i$ and all $i\in[N]$;
  \item $f_t:\mathbb{R}^n\to\mathbb{R}$ is convex for all $t\geq 0$;
  \item $\nabla_1 f_t(x,y)$ and $\nabla_2 f_t(x,y)$ are uniformly $L_1$-Lipschitz continuous, i.e., $\|\nabla_1 f_t(x,y)-\nabla_1 f_t(x',y')\|\leq L_1(\|x-x'\|+\|y-y'\|)$ and $\|\nabla_2 f_t(x,y)-\nabla_2 f_t(x',y')\|\leq L_1(\|x-x'\|+\|y-y'\|)$ for all $x,x'\in X$, $y,y'\in\mathbb{R}^{Nd}$, and all $t\geq 0$;
  \item $\nabla_1 f_t(x,y), \nabla_2 f_t(x,y)$ and $\nabla\psi_i(x_i)$ are uniformly bounded by a constant $G>0$;
  \item $\nabla\psi_i$ is $L_2$-Lipschitz continuous, i.e., $\|\nabla\psi_i(x)-\nabla\psi_i(x')\|\leq L_2\|x-x'\|$ for any $x,x'\in X_i$ and all $i\in[N]$.
\end{enumerate}
\end{assumption}

It is worth pointing out that $f_{i,t}$'s are not necessary to be convex, instead it is sufficient for $f_t$ to be convex. Moreover, it should be noted that the convexity and compactness of $X_i$'s have been utilized in lots of existing works on (distributed) online optimization, such as \cite{zinkevich2003online,hazan2007logarithmic,shalev2012online,koppel2015saddle,li2018distributedon,yi2019distributed}, to just name a few.

\subsection{Graph Theory}\label{s2.1}

Each agent must send its information to its out-neighbors in order to solve the global problem (\ref{1}). The communication pattern among all agents is described by a simple time-varying graph, denoted by $\mathcal{G}_t=(\mathcal{V},\mathcal{E}_t)$ with the node/agent set $\mathcal{V}=\{1,\ldots,N\}$ and the edge set $\mathcal{E}_t\subset\mathcal{V}\times\mathcal{V}$. An edge $(j,i)\in\mathcal{E}_t$ means that agent $j$ can send information to agent $i$ at time step $t$, where $j$ (resp. $i$) is called an in-neighbor (resp. out-neighbor) of $i$ (resp. $j$). Denote by $\mathcal{N}_{i,t}=\{j:(j,i)\in\mathcal{E}_t\}$ the in-neighbor set of node $i$ at time $t$. The graph $\mathcal{G}_t$ is called undirected if and only if $(i,j)\in\mathcal{E}_t$ amounts to $(j,i)\in\mathcal{E}_t$ for all $t\geq 0$, and directed otherwise. The communication matrix $A=(a_{ij,t})\in\mathbb{R}^{N\times N}$ is defined by: $a_{ij,t}>0$ if $(j,i)\in\mathcal{E}_t$, and $a_{ij,t}=0$ otherwise.

A few frequently used assumptions in the literature are listed below.

\begin{assumption}\label{a2}
The following hold for the communication graphs:
\begin{enumerate}
  \item $\mathcal{G}_t$ is $Q$-strongly connected for a constant $Q>0$, i.e., the union graph $(\mathcal{V},\cup_{l=0,\ldots,Q-1}\mathcal{E}_{k+l})$ is strongly connected for all $k\geq 0$;
  \item $A_t$ is doubly stochastic, i.e., $\sum_{j=1}^N a_{ij,t}=1$ and $\sum_{i=1}^N a_{ij,t}=1$ for all $i,j\in[N]$;
  \item There exists a constant $a\in(0,1)$ such that $a_{ij}\geq a$ whenever $a_{ij}>0$, and $a_{ii}\geq a$ for all $i\in[N]$.
\end{enumerate}
\end{assumption}

\section{Main Results}\label{s3}

This section provides the proposed algorithms and theoretical analysis, including two parts: 1) the case with true gradients, and 2) the case with stochastic/noisy gradients.

\subsection{The Case with True Gradients}

To handle problem (\ref{1}), the centralized projected gradient descent algorithm can be given as
\begin{align}
x_{t+1}=P_X(x_t-\alpha_t\nabla f_t(x_t)),          \label{gd1}
\end{align}
where $x_t=col(x_{1,t},\ldots,x_{N,t})$ and $\alpha_t$ is the stepsize. For each agent $i\in[N]$, (\ref{gd1}) can be written as
\begin{align}
x_{i,t+1}&=P_{X_i}\Big[x_t-\alpha_t\Big(\nabla_1 f_{i,t}(x_{i,t},\nu(x_t))         \nonumber\\
&\hspace{0.8cm}+\nabla\psi_i(x_{i,t})\frac{\sum_{i=1}^N\nabla_2 f_{i,t}(x_{i,t},\nu(x_t))}{N}\Big)\Big].       \label{gd2}
\end{align}

However, it is easy to observe that $\nu(x_t)$ and $\frac{1}{N}\sum_{i=1}^N\nabla_2 f_{i,t}(x_{i,t},\nu(x_t))$ are global information, which cannot be accessed by any individual agent. Thus, auxiliary variables must be introduced to track $\nu(x_t)$ and $\frac{1}{N}\sum_{i=1}^N\nabla_2 f_{i,t}(x_{i,t},\nu(x_t))$ in a distributed manner. To do so, we introduce $\nu_{i,t}$ and $y_{i,t}$ to track $\nu(x_t)$ and $\frac{1}{N}\sum_{i=1}^N\nabla_2 f_{i,t}(x_{i,t},\nu(x_t))$, respectively, for each agent $i$. Then, (\ref{gd2}) can be modified as (\ref{7-1}). Moreover, inspired by the idea of gradient tracking \cite{lee2017sublinear,li2018distributedon}, the updates of $\nu_{i,t}$ and $y_{i,t}$ are given in (\ref{7-2}) and (\ref{7-3}), respectively.

The developed online distributed algorithm is shown in Algorithm 1. At time step $t+1\geq 1$, each agent makes a decision $x_{i,t+1}$ according to (\ref{7-1}), and the function $f_{i,t+1}$ will be then revealed to agent $i$ along with its true gradients, followed by the update of $\nu_{i,t+1}$ and $y_{i,t+1}$. Please note that the terms $\sum_{j=1}^N a_{ij,t}\nu_{j,t}$ and $\sum_{j=1}^N a_{ij,t}y_{j,t}$ involve only local information exchanges, that is, agent $i$ has used the information $\nu_{j,t}$ and $y_{j,t}$ received from its in-neighbors $\{j: j\in[N],a_{ij,t}\neq 0\}$.

\begin{algorithm}
 \caption{Online Distributed Gradient Tracking (O-DGT) with True Gradients}
 \begin{algorithmic}[1]
  \STATE \textbf{Initialization:} Stepsize $\alpha_t$ in (\ref{11}), and local initial conditions $x_{i,0}\in X_i$, $\nu_{i,0}=\psi_i(x_{i,0})$, and $y_{i,0}=\nabla_2 f_{i,0}(x_{i,0},\nu_{i,0})$ for all $i\in[N]$.
  \STATE \textbf{Iterations:} Step $t\geq 0$: update for each $i\in[N]$:
\begin{subequations}
\begin{align}
x_{i,t+1}&=P_{X_i}[x_{i,t}-\alpha_t (\nabla_1 f_{i,t}(x_{i,t},\nu_{i,t})             \nonumber\\
&\hspace{2.5cm}+\nabla\psi_i(x_{i,t})y_{i,t})],                                           \label{7-1}\\
\nu_{i,t+1}&=\sum_{j=1}^N a_{ij,t}\nu_{j,t}+\psi_i(x_{i,t+1})-\psi_i(x_{i,t}),            \label{7-2}\\
y_{i,t+1}&=\sum_{j=1}^N a_{ij,t}y_{j,t}+\nabla_2 f_{i,t+1}(x_{i,t+1},\nu_{i,t+1})            \nonumber\\
&\hspace{2.2cm}-\nabla_2 f_{i,t}(x_{i,t},\nu_{i,t}).                                           \label{7-3}
\end{align}\label{7}
\end{subequations}
 \end{algorithmic}
\end{algorithm}

We are now in a position to present the main result on Algorithm 1.

\begin{theorem}\label{t1}
Under Assumptions \ref{a1} and \ref{a2}, let $\alpha_0=1$ and
\begin{align}
\alpha_t=\frac{1}{\sqrt{t}},~~~~~\text{for}~t\geq 1,         \label{11}
\end{align}
then there holds
\begin{align}
\mathcal{R}_T=O(\sqrt{T})+O(V_{T,\alpha_t^{-1}}^p)+O(\max\{V_T^g,V_{T,\alpha_t}^g\}),              \label{12}
\end{align}
where
{\small\begin{align}
&V_{T,\alpha_t^{-1}}^p:=\sum_{t=1}^T\frac{1}{\alpha_t}\|x_{t+1}^*-x_t^*\|,        \label{5}\\
&\hspace{-0.1cm}V_T^g:=\sum_{t=1}^T\sum_{i=1}^N\max_{\substack{x_i\in X_i\\ z_i\in\mathbb{R}^d}}\|\nabla_2f_{i,t+1}(x_i,z_i)-\nabla_2f_{i,t}(x_i,z_i)\|,      \label{6}\\
&V_{T,\alpha_t}^g:=\sum_{t=1}^T\alpha_t\Big(\sum_{i=1}^N\max_{\substack{x_i\in X_i\\ z_i\in\mathbb{R}^d}}\|\nabla_2f_{i,t+1}(x_i,z_i)-\nabla_2f_{i,t}(x_i,z_i)\|\Big)^2,      \label{6b}
\end{align}}are called {\em $\alpha_t^{-1}$-weighted path variation}, {\em gradient variation}, and {\em $\alpha_t$-weighted squared gradient variation}, respectively.

\end{theorem}
\begin{proof}
The proof can be found in Appendix A.
\end{proof}

\begin{remark}\label{r2}
To our best knowledge, this paper is the first to investigate problem (\ref{1}) with an aggregative variable. An algorithm has been devised for handling this problem with guaranteed dynamic regret. Besides, it is well known that achieving a sublinear bound on the dynamic regret is impossible in the worst case, unless some regularity measure is introduced for the sequence of loss functions \cite{yang2016tracking}, which is why $V_{T,\alpha_t^{-1}}^p$, $V_T^g$ and $V_{T,\alpha_t}^g$ are introduced, all representing the difference of $f_t$'s or $f_{i,t}$'s. Here, the gradient variation $V_T^g$ and $V_{T,\alpha_t}^g$ are new terms needed to bound the dynamic regret, which is generally unnecessary for the case without the aggregative variable \cite{yi2019distributed}. The reason behind this phenomenon is that a sequence of global time-varying gradients $\frac{1}{N}\sum_{i=1}^N \nabla_2 f_{i,t}(x_{i,t},\nu_{i,t})$ are unavailable to all agents and thus need to be estimated by all agents, i.e., $y_{i,t}$'s. Notice also that $V_{T,\alpha_t^{-1}}^p\leq \sqrt{T}V_{T,1}^p$.
\end{remark}

\begin{remark}\label{ra1}
Note that all other parameters independent of $T$ have been omitted in Theorem \ref{t1}. In fact, some parameters pertinent to the communication graph can be established, that is, (\ref{12}) can be more specifically provided as
\begin{align}
\mathcal{R}_T&=O\Big(\frac{N^2\sqrt{N}B_1\gamma\xi}{1-\xi}\sqrt{T}\Big)+O\big(\sqrt{N}V_{T,\alpha_t^{-1}}^p\big)            \nonumber\\
&\hspace{0.4cm}+O\Big(\frac{\sqrt{N}\gamma\xi}{1-\xi}V_T^g\Big)+O\Big(\frac{\gamma^2}{(1-\xi^2)^2}V_{T,\alpha_t}^g\Big),          \label{rae1}
\end{align}
where $B_1:=N\gamma\max_{i\in[N]}\|y_{i,1}\|+\frac{2NG\gamma\xi}{1-\xi}+4G$, $\gamma:=\big(1-\frac{a}{2N^2}\big)^{-2}$, and $\xi:=\big(1-\frac{a}{2N^2}\big)^{\frac{1}{Q}}$.

\end{remark}

As a special case, let us consider the scenario where all $f_t$'s are time-invariant, i.e., $f_t=f$ for some function $f:\mathbb{R}^n\to\mathbb{R}$ for all $t\geq 0$. In this case, Algorithm 1 is renamed a distributed gradient tracking algorithm (DGT) with true gradients. Then the following convergence result can be concluded.
\begin{corollary}\label{c1}
For the case with $f_t=f$ for all $t\geq 0$, if Assumptions \ref{a1} and \ref{a2} hold, along with $\alpha_t$ given in (\ref{11}), then there holds
\begin{align}
f(\bar{x}_T)-f(x^*)=O\Big(\frac{1}{\sqrt{T}}\Big),              \label{c1-1}
\end{align}
where
\begin{align}
\bar{x}_T:=\frac{1}{T}\sum_{t=1}^T x_t,~~~~~x^*&:=\mathop{\arg\min}_{x\in X} f(x).                  \label{c1-2}
\end{align}
\end{corollary}
\begin{proof}
It is easy to see that $V_{T,\alpha_t^{-1}}^p$, $V_T^g$, and $V_{T,\alpha_t}^g$ vanish in this case. Therefore, by Theorem 1, one has that
\begin{align}
\mathcal{R}_T=O(\sqrt{T}).              \label{c1-4}
\end{align}
In view of the definition of $\mathcal{R}_T$, it can be obtained that
\begin{align}
\frac{\mathcal{R}_T}{T}&=\sum_{t=1}^T \frac{1}{T}f_(x_t)-f(x^*)\geq f(\bar{x}_T)-f(x^*),           \label{c1-5}
\end{align}
where the inequality has used the convexity of $f$. Combining (\ref{c1-4}) with (\ref{c1-5}) can yield (\ref{c1-1}), thus ending the proof.
\end{proof}

It should be noted that a diminishing stepsize (\ref{11}) has been leveraged in Theorem \ref{t1}, and thus the upper bound in (\ref{12}) is not optimal. As seen in \cite{zhao2018proximal,yang2016tracking}, the optimal bound on dynamic regret is $O(\sqrt{T})+O(\sqrt{T}V_{T,1}^p)$. Along this line, a better result is provided below when using a constant stepsize.
\begin{theorem}\label{tmm2}
Under Assumptions \ref{a1} and \ref{a2}, let
\begin{align}
\alpha_t=\sqrt{\frac{1+V_{T,1}^p}{T+V_{T,1}^g}},~~~~~\text{for}~t\geq 0,         \label{ccc1}
\end{align}
then there holds
\begin{align}
\mathcal{R}_T&=O(\sqrt{T})+O(\sqrt{TV_{T,1}^p})+O(V_T^g)+O(\sqrt{V_{T,1}^p V_{T,1}^g}),              \label{ccc2}
\end{align}
where $V_{T,1}^p$, $V_T^g$, and $V_{T,1}^g$ are defined in (\ref{5})-(\ref{6b}).
\end{theorem}
\begin{proof}
When setting $\alpha_t=\alpha$ for all $t\geq 0$, where $\alpha>0$ is a constant, invoking the same arguments as that of Theorem~\ref{t1} can yield that
\begin{align}
\mathcal{R}_T&=O\Big(\frac{1}{\alpha}\Big)+O(\alpha T)+O\Big(\frac{V_{T,1}^p}{\alpha}\Big)+O(V_T^g)+O(\alpha V_{T,1}^g).                             \nonumber
\end{align}
By choosing $\alpha=\sqrt{(1+V_{T,1}^p)/(T+V_{T,1}^g)}$, the result in (\ref{ccc2}) can be directly obtained.
\end{proof}

\begin{remark}\label{ra2}
Note that distributed algorithms are studied in this paper. The bound in Theorem \ref{tmm2} is almost as good as the centralized algorithms in \cite{zhao2018proximal,yang2016tracking}, but a drawback is the requirement of knowing $T$, $V_{T,1}^p$, and $V_{T,1}^g$ beforehand, which also appears in \cite{zhao2018proximal,yang2016tracking}. In comparison, the stepsize in Theorem~\ref{t1} does not require any knowledge of $T$, $V_{T,1}^p$, and $V_{T,1}^g$, but at the cost of a more conservative regret bound.
\end{remark}

\subsection{The Case with Stochastic Gradients}

In this subsection, true gradients in Algorithm 1 are replaced with stochastic ones, which can be obtained by mini-batch samples. In this case, denote by $\tilde{\nabla}$ the stochastic gradient. The algorithm is given in Algorithm 2.

To proceed, it is standard to list a few assumptions on stochastic gradients, i.e., unbiased gradients and bounded variances.
\begin{assumption}\label{a3}
There exist constants $\sigma_1,\sigma_2>0$ such that for all $t\geq 0$,
\begin{align}
&\mathbb{E}[\tilde{\nabla}_1 f_{i,t}(x_{i,t},\nu_{i,t})|x_{i,t},\nu_{i,t}]=\nabla_1 f_{i,t}(x_{i,t},\nu_{i,t}),         \label{14}\\
&\mathbb{E}[\tilde{\nabla}\psi_i(x_{i,t})|x_{i,t}]=\nabla\psi_i(x_{i,t}),                                     \label{15}\\
&\mathbb{E}[\tilde{\nabla}_2 f_{i,t}(x_{i,t},\nu_{i,t})|x_{i,t},\nu_{i,t}]=\nabla_2 f_{i,t}(x_{i,t},\nu_{i,t}),         \label{16}\\
&\mathbb{E}[\|\tilde{\nabla}-\nabla\|^2|x_{i,t},\nu_{i,t}]\leq \sigma_1^2,                                  \label{17}\\
&\mathbb{E}[\|\tilde{\nabla}_2-\nabla_2\|^2|x_{i,t},\nu_{i,t}]\leq \sigma_2^2,                                          \label{17b}
\end{align}
where $\tilde{\nabla}$ stands for the stochastic gradients in (\ref{14}) and (\ref{15}), $\nabla$ is the corresponding true gradient, and $\tilde{\nabla}_2,\nabla_2$ denote $\tilde{\nabla}_2 f_{i,t}(x_{i,t},\nu_{i,t})$ and $\nabla_2 f_{i,t}(x_{i,t},\nu_{i,t})$, respectively.
\end{assumption}

\begin{algorithm}
 \caption{O-DGT with Stochastic Gradients}
 \begin{algorithmic}[1]
  \STATE \textbf{Initialization:} Stepsize $\alpha_t$ in (\ref{11}), and local initial conditions $x_{i,0}\in X_i$, $\nu_{i,0}=\psi_i(x_{i,0})$, and $y_{i,0}=\tilde{\nabla}_2 f_{i,0}(x_{i,0},\nu_{i,0})$ for all $i\in[N]$.
  \STATE \textbf{Iterations:} Step $t\geq 0$: update for each $i\in[N]$ by (\ref{7}) with true gradients $\nabla_1 f_{i,t}$, $\nabla\psi_i$, $\nabla_2 f_{i,t+1}$, and $\nabla_2 f_{i,t}$ being replaced with stochastic gradients $\tilde{\nabla}_1 f_{i,t}$, $\tilde{\nabla}\psi_i$, $\tilde{\nabla}_2 f_{i,t+1}$, and $\tilde{\nabla}_2 f_{i,t}$, respectively.
 \end{algorithmic}
\end{algorithm}

In this scenario, the dynamic regret in (\ref{4}) should be redefined in the sense of expectation, i.e., $\mathbb{E}(\mathcal{R}_T)$.

It is now ready to present the main result on Algorithm 2.
\begin{theorem}\label{t2}
Under Assumptions \ref{a1}-\ref{a3}, let $\alpha_t$ be the same as in Theorem \ref{t1}, then there holds
\begin{align}
\hspace{-0.1cm}\mathbb{E}(\mathcal{R}_T)=O(\sqrt{T})+O(V_{T,\alpha_t^{-1}}^p)+O(\max\{V_T^g,V_{T,\alpha_t}^g\}).              \nonumber
\end{align}
\end{theorem}
\begin{proof}
The proof can be found in Appendix B.
\end{proof}

\begin{remark}\label{r3}
It is worth mentioning that the result in Theorem \ref{t2} is the same as in the full gradient case. In addition, the similar bound to (\ref{rae1}) can also be derived by similar arguments, but some constants omitted in $O(\cdot)$ are proportional to $\sigma_1,\sigma_1\sigma_2$, and $\sigma_2^2$ as well. Moreover, similar to Theorem \ref{tmm2}, the same bound (\ref{ccc2}) can be obtained when applying the constant stepsize (\ref{ccc1}).
\end{remark}

To end this section, a similar result to Corollary \ref{c1} can be obtained for the time-invariant case, as shown below, where Algorithm 2 becomes a distributed gradient tracking algorithm (DGT) with stochastic gradients.

\begin{corollary}\label{c2}
For the case with $f_t=f$ for all $t\geq 0$, if Assumptions \ref{a1}-\ref{a3} hold, and $\alpha_t$ is the same as in Theorem \ref{t1}, then there holds
\begin{align}
\mathbb{E}[f(\bar{x}_T)]-f(x^*)=O\Big(\frac{1}{\sqrt{T}}\Big),              \label{c2-1}
\end{align}
where $\bar{x}_T$ and $x^*$ are defined in (\ref{c1-2}).
\end{corollary}
\begin{proof}
This corollary is a direct implication of Theorem \ref{t2} and the argument of Corollary \ref{c1}, and it is thus omitted.
\end{proof}

\section{A Numerical Example}\label{s4}

This section aims at providing a numerical example to corroborate the obtained theoretical results. In doing so, motivated by a simple synthetic robotics reactive control task in \cite{sun2017safety}, let us consider a target surrounding problem for robots in the plane, where there are $N$ robots (or agents), whose purpose is to protect a target, denoted by $x_0(t)$, by surrounding this target in order to avoid the attack from $M$ intruders. Also, each agent is only aware of some intruders, instead of all intruders.

For this problem, let $X_i=\mathbb{R}^2$, $N=M=50$, $Q=4$, and $\psi_i=Id$ for all $i\in[N]$. Without loss of generality, assume that agent $i\in[N]$ is only aware of the intruder $i\in[M]$, as shown in Fig. \ref{f0}. In this case, after each agent $i\in[N]$ decides to move to the position $x_{i,t}$ at time step $t\geq 0$, a loss will be incurred for agent $i$ proportional to the distance from $x_{i,t}$ to the intruder $i$ and the distance from the average position $\nu(x_t)$ of all agents to the target $x_0(t)$, i.e., $f_{i,t}(x_{i,t},\nu(x_t))=\|x_{i,t}-z_i(t)\|+\|\nu(x_t)-x_0(t)\|$, where $x_t=col(x_{1,t},\ldots,x_{N,t})$ and $z_i(t)$ represents the $i$-th intruder for all $i\in[M]$.

In the numerical simulation, set $x_0(t)=col(10,10)+1/(t+1)\cdot col(1,1)$, and $z_i(t)=col(10,10)+6\cdot col(\sin(t),\cos(t))+1/(t+1)\cdot col(1,1)$ for all $i\in[M]$. By running the developed Algorithms 1 and 2, the evolutions of dynamic regret are plotted in Fig. \ref{f1}, from which one can observe that the dynamic regret with true gradients, i.e., Algorithm 1, decreases faster than the expected dynamic regret with stochastic gradients, i.e., Algorithm 2, where gradients are stochastic with $\sigma_1^2=\sigma_2^2=0.1$. Meanwhile, the (expected)

\begin{figure}[H]
\centering
\includegraphics[width=1.9in]{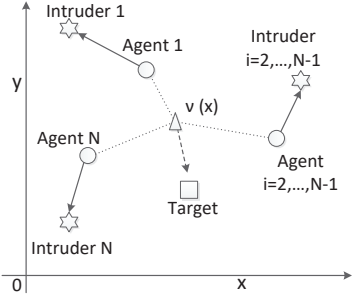}
\caption{Schematic illustration for the target surrounding problem, where solid arrows mean that agents are aware of the pointed intruders and the triangle represents the average $\nu(x)$ that aims to track the target in order to protect it from intruders' attack in the sense of almost surrounding it by all agents.}
\label{f0}
\end{figure}
\vspace{-0.5cm}
\begin{figure}[H]
\centering
\includegraphics[width=3.0in]{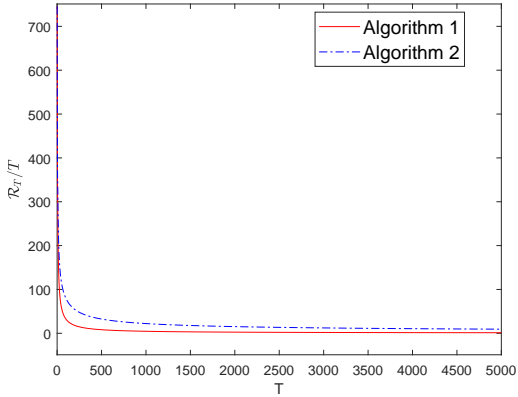}
\caption{Evolutions of $\mathcal{R}_T/T$ and $\mathbb{E}(\mathcal{R}_T)/T$ for Algorithms 1 and 2, respectively.}
\label{f1}
\end{figure}dynamic regrets for both cases in Fig. \ref{f1} tend to decrease asymptotically. In summary, the numerical simulations support the theoretical results on the proposed algorithms.

\section{Conclusion}\label{s5}

This paper studied distributed online convex optimization with an aggregative variable over a multi-agent network consisting of $N$ agents, where each agent must make a decision based on its own partial information on its local loss function and decision variable. All agents are required to cooperate by local information exchange among neighboring agents in order to tackle a global decision making problem at each time instant. To our best knowledge, this paper is the first to consider the dependency on an aggregative variable for all local loss functions, and the aggregative variable is unavailable to all agents. To cope with this problem, both true and stochastic/noisy gradients were taken into account, for which a novel algorithm, called online distributed gradient tracking (O-DGT), was developed based on true or stochastic gradients. It was shown that the dynamic regret is sublinear when some path and gradient variation terms are all sublinear in both the true and stochastic gradient cases. A numerical simulation was also provided to corroborate the developed algorithms. Future research directions can be placed on addressing unbalanced communication graphs (i.e., $A_t$ in Assumption \ref{a2} is only row- or column-stochastic) and asynchronous algorithms.

%

\section*{Appendix}

To facilitate the ensuing analysis, it is helpful to introduce some notations. For a vector $x=col(x_1,\ldots,x_N)\in\mathbb{R}^n$, let us define $\psi(x):=col(\psi_1(x_1),\ldots,\psi_N(x_N))$. For a vector-valued differentiable function $g(x)=col(g_1(x),\ldots,g_m(x))$, where $g_i$ is a real-valued function for all $i\in[m]$, denote $\nabla g(x)=(\nabla g_1(x),\ldots,\nabla g_m(x))$. Also, denote $P_X(z)=col(P_{X_1}(z_1),\ldots,P_{X_N}(z_N))$ for a vector $z=col(z_1,\ldots,z_N)\in\mathbb{R}^n$.

With the above notations, algorithm (\ref{7}) can be written in a compact form as
\begin{align}
&x_{t+1}=P_X[x_t-\alpha_t(\nabla_1 f_t(x_t,\nu_t)+\nabla\psi(x_t)y_t)],            \label{8}\\
&\nu_{t+1}=\mathcal{A}_t\nu_t+\psi(x_{t+1})-\psi(x_t),                               \label{9}\\
&\hspace{-0.2cm}y_{t+1}=\mathcal{A}_ty_t+\nabla_2f_{t+1}(x_{t+1},\nu_{t+1})-\nabla_2f_t(x_t,\nu_t),       \label{10}
\end{align}
where$\mathcal{A}_t:=(A_t\otimes I_d)$, $x_t:=col(x_{1,t},\ldots,x_{N,t})$ ($\nu_t$ and $y_t$ are similarly defined), and $\nabla_1f_t,\nabla_2f_t$ are defined in the paragraph after Example \ref{e1}.

\subsection{Proof of Theorem \ref{t1}}

To begin with, several lemmas are first provided.

\begin{lemma}\label{l1}
Under Assumption \ref{a2}, for all $t\geq 0$, there holds
\begin{align}
\bar{\nu}_t&:=\frac{1}{N}\sum_{i=1}^N \nu_{i,t}=\frac{1}{N}\sum_{i=1}^N \psi_i(x_{i,t})=\nu(x_t),        \label{ap1}\\
\bar{y}_t&:=\frac{1}{N}\sum_{i=1}^N y_{i,t}=\frac{1}{N}\sum_{i=1}^N\nabla_2 f_{i,t}(x_{i,t},\nu_{i,t}).       \label{ap2}
\end{align}
\end{lemma}
\begin{proof}
In light of column-stochasticity of $A_t$ in Assumption \ref{a2}, multiplying $\textbf{1}^\top/N$ on both sides of (\ref{9}) yields that
\begin{align}
\bar{\nu}_{t+1}=\bar{\nu}_t+\frac{1}{N}\sum_{i=1}^N \psi_i(x_{i,t+1})-\frac{1}{N}\sum_{i=1}^N \psi_i(x_{i,t}),          \nonumber
\end{align}
which further implies that
\begin{align}
\bar{\nu}_t-\frac{1}{N}\sum_{i=1}^N \psi_i(x_{i,t})=\bar{\nu}_0-\frac{1}{N}\sum_{i=1}^N \psi_i(x_{i,0}).       \nonumber
\end{align}
Note that $\nu_{i,0}=\psi_i(x_{i,0})$. Thus, $\bar{\nu}_0-\frac{1}{N}\sum_{i=1}^N \psi_i(x_{i,0})=0$, which together with the above equality leads to (\ref{ap1}).

The assertion (\ref{ap2}) can be similarly proved as above.
\end{proof}

\begin{lemma}\label{l2}
Under Assumptions \ref{a1} and \ref{a2}, there holds
\begin{align}
\|y_t-{\bf 1}_N\otimes \bar{y}_t\|&\leq NB_1,                              \label{ap3}\\
\|x_{t+1}-x_t\|&\leq G(1+G+NB_1)\alpha_t,~~~~\forall t\geq 0            \label{ap4}
\end{align}
where $B_1:=N\gamma\max_{i\in[N]}\|y_{i,1}\|+\frac{2NG\gamma\xi}{1-\xi}+4G$, and
\begin{align}
\gamma:=\big(1-\frac{a}{2N^2}\big)^{-2},~~~\xi:=\big(1-\frac{a}{2N^2}\big)^{\frac{1}{Q}}.           \label{ap5}
\end{align}
\end{lemma}
\begin{proof}
The iteration (\ref{7-3}) can be rewritten as
\begin{align}
y_{i,t+1}=\sum_{j=1}^N a_{ij,t}y_{j,t}+\epsilon_{i,t+1}^y,            \label{ap6}
\end{align}
where $\epsilon_{i,t+1}^y:=\nabla_2 f_{i,t+1}(x_{i,t+1},\nu_{i,t+1})-\nabla_2 f_{i,t}(x_{i,t},\nu_{i,t})$. Invoking Lemma 2 in \cite{lee2017sublinear} for (\ref{ap6}) can lead to
\begin{align}
\|y_{i,t+1}-\bar{y}_{t+1}\|&\leq N\gamma\xi^t\max_{i\in[N]}\|y_{i,1}\|+\gamma\sum_{l=1}^{t-1}\xi^{t-l}\sum_{j=1}^N\|\epsilon_{j,l+1}^y\|         \nonumber\\
&\hspace{0.4cm}+\frac{1}{N}\sum_{j=1}^N \|\epsilon_{j,t+1}^y\|+\|\epsilon_{i,t+1}^y\|.         \label{ap7}
\end{align}

By invoking Assumption \ref{a1}.4, it can be obtained that
\begin{align}
\|\epsilon_{i,t+1}^y\|&\leq \|\nabla_2 f_{i,t+1}(x_{i,t+1},\nu_{i,t+1})\|+\|\nabla_2 f_{i,t}(x_{i,t},\nu_{i,t})\|       \nonumber\\
&\leq 2G.           \label{ap8}
\end{align}
In view of (\ref{ap7}) and (\ref{ap8}), one can obtain that $\|y_{i,t+1}-\bar{y}_{t+1}\|\leq B_1$, which further results in
\begin{align}
\|y_t-{\bf 1}_N\otimes\bar{y}_t\|\leq\sum_{i=1}^N\|y_{i,t}-\bar{y}_t\|\leq NB_1.         \nonumber
\end{align}

For (\ref{ap4}), with reference to (\ref{2}) and (\ref{8}), one has that
\begin{align}
\|x_{t+1}-x_t\|&\leq \alpha_t\|\nabla_1 f_t(x_t,\nu_t)+\nabla\psi(x_t)y_t\|         \nonumber\\
&\leq\alpha_t\|\nabla_1 f_t(x_t,\nu_t)\|          \nonumber\\
&\hspace{0.4cm}+\alpha_t\|\nabla\psi(x_t)\|\|{\bf 1}_N\otimes \bar{y}_t\|       \nonumber\\
&\hspace{0.4cm}+\alpha_t\|\nabla\psi(x_t)\|\|y_t-{\bf 1}_N\otimes \bar{y}_t\|.            \label{ap9}
\end{align}
Meanwhile, by (\ref{ap2}), it is easy to obtain that
\begin{align}
\|{\bf 1}_N\otimes \bar{y}_t\|^2&=\|{\bf 1}_N\otimes \frac{1}{N}\sum_{i=1}^N\nabla_2f_{i,t}(x_{i,t},\nu_{i,t})\|^2         \nonumber\\
&=\frac{1}{N}\|\sum_{i=1}^N\nabla_2f_{i,t}(x_{i,t},\nu_{i,t})\|^2        \nonumber\\
&\leq \sum_{i=1}^N\|\nabla_2f_{i,t}(x_{i,t},\nu_{i,t})\|^2               \nonumber\\
&=\|\nabla_2f_t(x_t,\nu_t)\|^2         \nonumber\\
&\leq G^2,                          \label{ap10}
\end{align}
where the first inequality has employed the fact $\|\sum_{i=1}^N z_i\|^2\leq N\sum_{i=1}^N \|z_i\|^2$ for any vectors $z_i$'s, and Assumption \ref{a1}.4 has been used in the last inequality. Applying (\ref{ap10}) and Assumption \ref{a1}.4 to (\ref{ap9}) can lead to the assertion (\ref{ap4}). This ends the proof.
\end{proof}

\begin{lemma}\label{l3}
Under Assumptions \ref{a1} and \ref{a2}, there holds
\begin{align}
&\sum_{t=1}^T\|\nu_t-{\bf 1}_N\otimes \bar{\nu}_t\|=O\Big(\frac{N\sqrt{N}B_1\gamma\xi}{1-\xi}\sum_{t=1}^T\alpha_t\Big),               \label{ap11}\\
&\sum_{t=1}^T\alpha_t\|\nu_t-{\bf 1}_N\otimes \bar{\nu}_t\|^2=O\Big(\frac{N^3B_1^2\gamma^2}{(1-\xi^2)^2}\sum_{t=1}^T\alpha_t^3\Big),               \label{xx1}\\
&\sum_{t=1}^T\|y_t-{\bf 1}_N\otimes \bar{y}_t\|=O\Big(\frac{N^2B_1\gamma\xi}{1-\xi}\sum_{t=1}^T\alpha_t\Big)           \nonumber\\
&\hspace{3.3cm}+O\Big(\frac{\gamma\xi}{1-\xi}V_T^g\Big),                             \label{ap12}\\
&\sum_{t=1}^T\alpha_t\|y_t-{\bf 1}_N\otimes \bar{y}_t\|^2=O\Big(\frac{N^4B_1^2\gamma^4}{(1-\xi^2)^4}\sum_{t=1}^T\alpha_t^3\Big)           \nonumber\\
&\hspace{3.8cm}+O\Big(\frac{\gamma^2}{(1-\xi^2)^2}V_{T,\alpha_t}^g\Big).             \label{xx2}
\end{align}
\end{lemma}
\begin{proof}
It is easy to see that (\ref{7-2}) can be rewritten as
\begin{align}
\nu_{i,t+1}=\sum_{j=1}^N a_{ij,t}\nu_{j,t}+\epsilon_{i,t+1}^\nu,           \label{ap13}
\end{align}
where $\epsilon_{i,t+1}^\nu:=\psi_i(x_{i,t+1})-\psi_i(x_{i,t})$. For (\ref{ap13}), invoking Lemma 2 in \cite{lee2017sublinear} can obtain that
\begin{align}
\|\nu_{i,t+1}-\bar{\nu}_{t+1}\|&\leq N\gamma\xi^t\max_{i\in[N]}\|\nu_{i,1}\|+\gamma\sum_{l=1}^{t-1}\xi^{t-l}\sum_{j=1}^N\|\epsilon_{j,l+1}^\nu\|         \nonumber\\
&\hspace{0.4cm}+\frac{1}{N}\sum_{j=1}^N \|\epsilon_{j,t+1}^\nu\|+\|\epsilon_{i,t+1}^\nu\|.         \label{ap14}
\end{align}
Appealing to $\|\nabla\psi_i(x_i)\|\leq G$ in Assumption \ref{a1}.4, it can be concluded that
\begin{align}
\|\epsilon_{i,t+1}^\nu\|\leq G\|x_{i,t+1}-x_{i,t}\|,            \nonumber
\end{align}
which implies that
\begin{align}
\sum_{i=1}^N\|\epsilon_{i,t+1}^\nu\|&\leq G\sum_{i=1}^N\|x_{i,t+1}-x_{i,t}\|         \nonumber\\
&\leq \sqrt{N}G\|x_{t+1}-x_t\|        \nonumber\\
&\leq \sqrt{N}G^2(1+G+NB_1)\alpha_t,            \label{ap15}
\end{align}
where the second inequality has exploited the fact that $\sum_{i=1}^N \|z_i\|\leq \sqrt{N}\sqrt{\sum_{i=1}^N \|z_i\|^2}$ for any vectors $z_i$'s, and (\ref{ap4}) has been used in the last inequality.

Moreover, for a sequence $\{\beta_t\}$, it is easy to verify that
\begin{align}
\sum_{t=1}^T\sum_{l=1}^{t-1}\xi^{t-l}\beta_{l+1}=\sum_{l=1}^{T-1}\xi^l\sum_{t=2}^{T+1-l}\beta_t\leq\frac{\xi}{1-\xi}\sum_{t=1}^T\beta_t.          \label{ap16}
\end{align}

Now, substituting (\ref{ap15}) and (\ref{ap16}) into (\ref{ap14}), together with
\begin{align}
\sum_{t=1}^T\|\nu_t-{\bf 1}_N\otimes\bar{\nu}_t\|\leq\sum_{t=1}^T\sum_{i=1}^N\|\nu_{i,t}-\bar{\nu}_t\|,       \nonumber
\end{align}
finishes the proof of (\ref{ap11}).

For (\ref{xx1}), invoking (\ref{ap14}) leads to
\begin{align}
&\|\nu_{i,t+1}-\bar{\nu}_{t+1}\|^2        \nonumber\\
&\leq 4N^2\gamma^2\xi^{2t}\max_{i\in[N]}\|\nu_{i,1}\|^2+\frac{4}{N^2}\big(\sum_{j=1}^N \|\epsilon_{j,t+1}^\nu\|\big)^2         \nonumber\\
&\hspace{0.2cm}+4\gamma^2 t\sum_{l=1}^{t-1}\xi^{2(t-l)}\big(\sum_{j=1}^N\|\epsilon_{j,l+1}^\nu\|\big)^2+4\|\epsilon_{i,t+1}^\nu\|^2,         \label{xx3}
\end{align}
where the fact that $(s_1+\cdots+s_m)^2\leq m(s_1^2+\cdots+s_m^2)$ for $s_i\geq0,i\in[m]$ has been utilized.

Moreover, for a sequence $\{\beta_t\}$ and $\eta\in(0,1)$, it is easy to verify that
\begin{align}
\sum_{t=1}^T t\sum_{l=1}^{t-1}\eta^{t-l}\beta_{l+1}\leq S\sum_{t=1}^T\beta_t,         \label{xx4}
\end{align}
where $S:=\sum_{k=0}^\infty (k+1)\eta^k$. In addition, it can be obtained that $S-\eta S=\sum_{k=0}^\infty\eta^k=1/(1-\eta)$, thus implying $S=1/(1-\eta)^2$. Therefore, one has that
\begin{align}
\sum_{t=1}^T t\sum_{l=1}^{t-1}\eta^{t-l}\beta_{l+1}\leq \frac{1}{(1-\eta)^2}\sum_{t=1}^T\beta_t.         \label{xx5}
\end{align}

In light of the nonincreasing property of $\alpha_t$, it can be concluded that $\sum_{t=1}^T\alpha_t t\sum_{l=1}^{t-1}\eta^{t-l}\beta_{l+1}\leq \sum_{t=1}^T t\sum_{l=1}^{t-1}\eta^{t-l}(\alpha_{l+1}\beta_{l+1})$, which, combining with (\ref{xx3}), (\ref{ap15}) and (\ref{xx5}), results in (\ref{xx1}).

As for (\ref{ap12}), bearing (\ref{ap6}) in mind, it can be obtained that
\begin{align}
\|\epsilon_{i,t+1}^y\|&=\|\nabla_2 f_{i,t+1}(x_{i,t+1},\nu_{i,t+1})-\nabla_2 f_{i,t}(x_{i,t},\nu_{i,t})\|         \nonumber\\
&\leq \|\nabla_2 f_{i,t}(x_{i,t+1},\nu_{i,t+1})-\nabla_2 f_{i,t}(x_{i,t},\nu_{i,t})\|                     \nonumber\\
&\hspace{-0.6cm}+\|\nabla_2 f_{i,t+1}(x_{i,t+1},\nu_{i,t+1})-\nabla_2 f_{i,t}(x_{i,t+1},\nu_{i,t+1})\|           \nonumber\\
&\leq L_1(\|x_{i,t+1}-x_{i,t}\|+\|\nu_{i,t+1}-\nu_{i,t}\|)                     \nonumber\\
&\hspace{-0.6cm}+\max_{x_i\in X_i,z_i\in\mathbb{R}^d}\|\nabla_2f_{i,t+1}(x_i,z_i)-\nabla_2f_{i,t}(x_i,z_i)\|,           \nonumber
\end{align}
where the last inequality has leveraged Assumption \ref{a1}.3. Therefore, one has that
\begin{align}
&\sum_{i=1}^N\|\epsilon_{i,t+1}^y\|\leq L_1(\sum_{i=1}^N\|x_{i,t+1}-x_{i,t}\|+\sum_{i=1}^N\|\nu_{i,t+1}-\nu_{i,t}\|)                     \nonumber\\
&\hspace{0.4cm}+\sum_{i=1}^N\max_{x_i\in X_i,z_i\in\mathbb{R}^d}\|\nabla_2f_{i,t+1}(x_i,z_i)-\nabla_2f_{i,t}(x_i,z_i)\|           \nonumber\\
&\leq \sum_{i=1}^N\max_{x_i\in X_i,z_i\in\mathbb{R}^d}\|\nabla_2f_{i,t+1}(x_i,z_i)-\nabla_2f_{i,t}(x_i,z_i)\|                     \nonumber\\
&\hspace{0.4cm}+\sqrt{N}L_1(\|x_{t+1}-x_t\|+\|\nu_{t+1}-\nu_t\|),           \label{ap17}
\end{align}
where the second inequality has employed the fact that $\sum_{i=1}^N \|z_i\|\leq \sqrt{N}\sqrt{\sum_{i=1}^N \|z_i\|^2}$ for any vectors $z_i$'s.

Let us now analyze the term $\|\nu_{t+1}-\nu_t\|$. In light of (\ref{9}), one has that
\begin{align}
&\|\nu_{t+1}-\nu_t\|          \nonumber\\
&=\|(\mathcal{A}_t-I)(I-\frac{1}{N}{\bf 1}_N{\bf 1}_N^\top\otimes I_d)\nu_t+\psi(x_{t+1})-\psi(x_t)\|     \nonumber\\
&\leq \|(A_t-I)\otimes I_d(\nu_t-{\bf 1}_N\otimes\bar{\nu}_t)\|+\|\psi(x_{t+1})-\psi(x_t)\|          \nonumber\\
&\leq \|A_t-I\|\|\nu_t-{\bf 1}_N\otimes\bar{\nu}_t\|+G\|x_{t+1}-x_t\|                                \nonumber\\
&\leq 2\|\nu_t-{\bf 1}_N\otimes\bar{\nu}_t\|+G\|x_{t+1}-x_t\|,                                       \label{ap18}
\end{align}
where the second and third inequalities have used Assumption \ref{a1}.4 and $\|A_t-I\|\leq 2$, respectively.

Inserting (\ref{ap18}) into (\ref{ap17}) gives rise to
\begin{align}
&\sum_{i=1}^N\|\epsilon_{i,t+1}^y\|         \nonumber\\
&\leq \sqrt{N}L_1(1+G)\|x_{t+1}-x_t\|+2\sqrt{N}L_1\|\nu_t-{\bf 1}_N\otimes\bar{\nu}_t\|        \nonumber\\
&\hspace{0.4cm}+\sum_{i=1}^N\max_{x_i\in X_i,z_i\in\mathbb{R}^d}\|\nabla_2f_{i,t+1}(x_i,z_i)-\nabla_2f_{i,t}(x_i,z_i)\|,          \nonumber
\end{align}
which, together with (\ref{ap7}) and $\sum_{t=1}^T\|y_t-{\bf 1}_N\otimes\bar{y}_t\|\leq \sum_{t=1}^T\sum_{i=1}^N\|y_{i,t}-\bar{y}_t\|$, leads to (\ref{ap12}).

Finally, (\ref{xx2}) can be similarly derived as (\ref{xx1}). This ends the proof.
\end{proof}

With the above preparations, it is now ready to prove Theorem \ref{t1}.

{\bf Proof of Theorem \ref{t1}:} By (\ref{8}) and (\ref{2}), one can obtain that
\begin{align}
\|x_{t+1}-x_t^*\|^2&\leq\|x_t-x_t^*-\alpha_t(\nabla_1f_t(x_t,\nu_t)+\nabla\psi(x_t)y_t)\|        \nonumber\\
&\hspace{-0.4cm}=\|x_t-x_t^*\|^2+\alpha_t^2\|\nabla_1f_t(x_t,\nu_t)+\nabla\psi(x_t)y_t\|^2          \nonumber\\
&\hspace{0.0cm}-2\alpha_t\langle x_t-x_t^*,\nabla_1f_t(x_t,\nu_t)+\nabla\psi(x_t)y_t\rangle.      \nonumber
\end{align}
Note that $\|\nabla_1f_t(x_t,\nu_t)+\nabla\psi(x_t)y_t\|^2=\|\nabla_1f_t(x_t,\nu_t)+\nabla\psi(x_t){\bf 1}_N\otimes\bar{y}_t+\nabla\psi(x_t)(y_t-{\bf 1}_N\otimes\bar{y}_t)\|^2\leq 2\|\nabla_1f_t(x_t,\nu_t)+\nabla\psi(x_t){\bf 1}_N\otimes\bar{y}_t\|^2+2\|\nabla\psi(x_t)\|^2\|(y_t-{\bf 1}_N\otimes\bar{y}_t)\|^2$. As a result, it can be deduced that
\begin{align}
&\|x_{t+1}-x_t^*\|^2        \nonumber\\
&\leq \|x_t-x_t^*\|^2+2\alpha_t^2\|\nabla_1f_t(x_t,\nu_t)+\nabla\psi(x_t){\bf 1}_N\otimes\bar{y}_t\|^2          \nonumber\\
&\hspace{0.4cm}+2\alpha_t^2\|\nabla\psi(x_t)\|^2\|(y_t-{\bf 1}_N\otimes\bar{y}_t)\|^2                              \nonumber\\
&\hspace{0.4cm}-2\alpha_t\langle x_t-x_t^*,\nabla\psi(x_t)(y_t-{\bf 1}_N\otimes\bar{y}_t)\rangle      \nonumber\\
&\hspace{0.4cm}-2\alpha_t\langle x_t-x_t^*,\nabla_1f_t(x_t,\nu_t)+\nabla\psi(x_t){\bf 1}_N\otimes\bar{y}_t\rangle,                  \nonumber
\end{align}
which, together with Assumption \ref{a1}.4 and (\ref{ap10}), leads to
\begin{align}
&\|x_{t+1}-x_t^*\|^2        \nonumber\\
&\leq \|x_t-x_t^*\|^2+2\alpha_t^2G^2(1+G)^2         \nonumber\\
&\hspace{0.4cm}+2\alpha_t^2G^2\|y_t-{\bf 1}_N\otimes\bar{y}_t\|^2          \nonumber\\
&\hspace{0.4cm}+2\alpha_t\|x_t-x_t^*\|\|\nabla\psi(x_t)\|\|y_t-{\bf 1}_N\otimes\bar{y}_t\|      \nonumber\\
&\hspace{0.4cm}-2\alpha_t\langle x_t-x_t^*,\nabla_1f_t(x_t,\nu_t)+\nabla\psi(x_t){\bf 1}_N\otimes\bar{y}_t\rangle         \nonumber\\
&\leq \|x_t-x_t^*\|^2+2\alpha_t^2G^2(1+G)^2          \nonumber\\
&\hspace{0.4cm}+2\alpha_t^2G^2\|y_t-{\bf 1}_N\otimes\bar{y}_t\|^2          \nonumber\\
&\hspace{0.4cm}+4\sqrt{N}BG\alpha_t\|y_t-{\bf 1}_N\otimes\bar{y}_t\|                                                  \nonumber\\
&\hspace{0.4cm}-2\alpha_t\langle x_t-x_t^*,\nabla_1f_t(x_t,\nu_t)+\nabla\psi(x_t){\bf 1}_N\otimes\bar{y}_t\rangle,                  \label{ap19}
\end{align}
where the last inequality has applied Assumption \ref{a1}.1, i.e., $\|x\|\leq \sqrt{N}B$ for any $x=col(x_1,\ldots,x_N)\in X$.

For the last term in (\ref{ap19}), by noting $\bar{\nu}_t=\nu(x_t)$ in Lemma \ref{l1}, it can be obtained that
\begin{align}
&-2\alpha_t\langle x_t-x_t^*,\nabla_1f_t(x_t,\nu_t)+\nabla\psi(x_t){\bf 1}_N\otimes\bar{y}_t\rangle           \nonumber\\
&=-2\alpha_t\langle x_t-x_t^*,\nabla_1f_t(x_t,{\bf 1}_N\otimes\bar{\nu}_t)         \nonumber\\
&\hspace{1.4cm}+\nabla\psi(x_t){\bf 1}_N\otimes\frac{1}{N}\sum_{i=1}^N\nabla_2f_{i,t}(x_{i,t},\bar{\nu}_t)\rangle     \nonumber\\
&\hspace{0.4cm}-2\alpha_t\langle x_t-x_t^*,\nabla_1f_t(x_t,\nu_t)-\nabla_1f_t(x_t,{\bf 1}_N\otimes\bar{\nu}_t)\rangle       \nonumber\\
&\hspace{0.4cm}-2\alpha_t\langle x_t-x_t^*,\nabla\psi(x_t){\bf 1}_N\otimes\frac{1}{N}\sum_{i=1}^N[\nabla_2f_{i,t}(x_{i,t},\nu_{i,t})   \nonumber\\
&\hspace{1.4cm}-\nabla_2f_{i,t}(x_{i,t},\bar{\nu}_t)]\rangle            \nonumber\\
&\leq 2\alpha_t[f_t(x_t^*)-f_t(x_t)]          \nonumber\\
&\hspace{0.4cm}+4\sqrt{N}B\alpha_t\|\nabla_1f_t(x_t,\nu_t)-\nabla_1f_t(x_t,{\bf 1}_N\otimes\bar{\nu}_t)\|       \nonumber\\
&\hspace{0.4cm}+4\sqrt{N}B\alpha_t\|\nabla\psi(x_t)\|\|{\bf 1}_N\otimes\frac{1}{N}\sum_{i=1}^N[\nabla_2f_{i,t}(x_{i,t},\nu_{i,t})   \nonumber\\
&\hspace{1.4cm}-\nabla_2f_{i,t}(x_{i,t},\bar{\nu}_t)]\|,                 \nonumber
\end{align}
where the convexity of $f_t$, the Cauchy-Schwarz inequality, and Assumption \ref{a1}.1 have been utilized in the last inequality. Applying Assumptions \ref{a1}.3 and \ref{a1}.4 to the above inequality can further yield that
\begin{align}
&-2\alpha_t\langle x_t-x_t^*,\nabla_1f_t(x_t,\nu_t)+\nabla\psi(x_t){\bf 1}_N\otimes\bar{y}_t\rangle           \nonumber\\
&\leq 2\alpha_t[f_t(x_t^*)-f_t(x_t)]        \nonumber\\
&\hspace{0.4cm}+4\sqrt{N}BL_1(1+G)\alpha_t\|\nu_t-{\bf 1}_N\otimes\bar{\nu}_t\|.                 \label{ap20}
\end{align}

Combining (\ref{ap19}) with (\ref{ap20}), one can obtain that
\begin{align}
&f_t(x_t)-f_t(x_t^*)\leq \frac{1}{2\alpha_t}(\|x_t-x_t^*\|^2-\|x_{t+1}-x_{t+1}^*\|^2)         \nonumber\\
&\hspace{0.4cm}+\frac{1}{2\alpha_t}(\|x_{t+1}-x_{t+1}^*\|^2-\|x_{t+1}-x_t^*\|^2)            \nonumber\\
&\hspace{0.4cm}+\alpha_tG^2(1+G)^2+\alpha_t G^2\|y_t-{\bf 1}_N\otimes\bar{y}_t\|^2                  \nonumber\\
&\hspace{0.4cm}+2\sqrt{N}BL_1(1+G)\|\nu_t-{\bf 1}_N\otimes \bar{\nu}_t\|           \nonumber\\
&\hspace{0.4cm}+2\sqrt{N}BG\|y_t-{\bf 1}_N\otimes\bar{y}_t\|,       \nonumber
\end{align}
which, by the summation over $t\in[T]$, leads to
\begin{align}
\mathcal{R}_T&\leq \sum_{t=1}^T\frac{1}{2\alpha_t}(\|x_t-x_t^*\|^2-\|x_{t+1}-x_{t+1}^*\|^2)         \nonumber\\
&\hspace{-0.4cm}+\sum_{t=1}^T\frac{1}{2\alpha_t}(\|x_{t+1}-x_{t+1}^*\|^2-\|x_{t+1}-x_t^*\|^2)            \nonumber\\
&\hspace{-0.4cm}+G^2(1+G)^2\sum_{t=1}^T\alpha_t+G^2\sum_{t=1}^T\alpha_t \|y_t-{\bf 1}_N\otimes\bar{y}_t\|^2                  \nonumber\\
&\hspace{-0.4cm}+2\sqrt{N}BL_1(1+G)\sum_{t=1}^T\|\nu_t-{\bf 1}_N\otimes \bar{\nu}_t\|           \nonumber\\
&\hspace{-0.4cm}+2\sqrt{N}BG\sum_{t=1}^T\|y_t-{\bf 1}_N\otimes\bar{y}_t\|.                    \label{ap21}
\end{align}

For the first term on the right-hand side of (\ref{ap21}), it is easy to calculate that
\begin{align}
&\sum_{t=1}^T\frac{1}{2\alpha_t}(\|x_t-x_t^*\|^2-\|x_{t+1}-x_{t+1}^*\|^2)    \nonumber\\
&=\frac{1}{2\alpha_1}\|x_{1}-x_{1}^*\|^2-\frac{1}{2\alpha_T}\|x_{T+1}-x_{T+1}^*\|^2       \nonumber\\
&\hspace{0.4cm}+\frac{1}{2}\sum_{t=2}^T\Big(\frac{1}{\alpha_t}-\frac{1}{\alpha_{t-1}}\Big)\|x_{t}-x_{t}^*\|^2       \nonumber\\
&\leq \frac{1}{\alpha_1}(\|x_{1}\|^2+\|x_{1}^*\|^2)+\sum_{t=2}^T\Big(\frac{1}{\alpha_t}-\frac{1}{\alpha_{t-1}}\Big)(\|x_{t}\|^2+\|x_{t}^*\|^2)       \nonumber\\
&\leq \frac{2NB^2}{\alpha_T},               \label{ap22}
\end{align}
where $\|z_1+z_2\|^2\leq 2(\|z_1\|^2+\|z_2\|^2)$ for any two vectors $z_1,z_2$ and $\frac{1}{\alpha_t}-\frac{1}{\alpha_{t-1}}>0$ have been used in the first inequality, and Assumption \ref{a1}.1 has been employed in the last inequality.

For the second term on the right-hand side of (\ref{ap21}), one has that
\begin{align}
&\sum_{t=1}^T\frac{1}{2\alpha_t}(\|x_{t+1}-x_{t+1}^*\|^2-\|x_{t+1}-x_t^*\|^2)    \nonumber\\
&=\sum_{t=1}^T\frac{1}{2\alpha_t}(x_{t+1}^*-x_{t}^*)^\top(x_{t}^*+x_{t+1}^*-2x_{t+1})    \nonumber\\
&\leq 2\sqrt{N}B\sum_{t=1}^T\frac{1}{\alpha_t}\|x_{t+1}^*-x_{t}^*\|            \nonumber\\
&= 2\sqrt{N}B V_{T,\alpha_t^{-1}}^p,         \label{ap23}
\end{align}
where the first inequality was resulted from the Cauchy-Schwarz inequality and Assumption \ref{a1}.1.

Finally, by appealing to (\ref{11}) and Lemma \ref{l3}, substituting (\ref{ap22})-(\ref{ap23}) into (\ref{ap21}) can complete the proof.
\hfill\rule{2mm}{2mm}

\subsection{Proof of Theorem \ref{t2}}

The proof is similar to that of Theorem \ref{t1}, but Lemmas \ref{l1}-\ref{l3} need to be modified in the stochastic scenario.

\begin{lemma}\label{l4}
Under Assumption \ref{a2}, for all $t\geq 0$, there holds
\begin{align}
\bar{\nu}_t&:=\frac{1}{N}\sum_{i=1}^N \nu_{i,t}=\frac{1}{N}\sum_{i=1}^N \psi_i(x_{i,t})=\nu(x_t),        \label{ap24}\\
\bar{y}_t&:=\frac{1}{N}\sum_{i=1}^N y_{i,t}=\frac{1}{N}\sum_{i=1}^N\tilde{\nabla}_2 f_{i,t}(x_{i,t},\nu_{i,t}).       \label{ap25}
\end{align}
\end{lemma}
\begin{proof}
The proof is the same as that of Lemma \ref{l1}.
\end{proof}

\begin{lemma}\label{l5}
Under Assumptions \ref{a1} and \ref{a2}, there holds
\begin{align}
\|y_t-{\bf 1}_N\otimes \bar{y}_t\|&=O(1),                              \label{ap26}\\
\mathbb{E}(\|x_{t+1}-x_t\|)&=O(\alpha_t),~~~~\forall t\geq 0.            \label{ap27}
\end{align}
\end{lemma}
\begin{proof}
It can be similarly proved to that of Lemma \ref{l2}, once noting that there holds
\begin{align}
\mathbb{E}(\|\tilde{\nabla}\||x_{i,t})&=\mathbb{E}(\|\nabla+\tilde{\nabla}-\nabla\||x_{i,t})      \nonumber\\
&\leq \|\nabla\|+\mathbb{E}(\|\tilde{\nabla}-\nabla\||x_{i,t})       \nonumber\\
&\leq \|\nabla\|+\sqrt{\mathbb{E}(\|\tilde{\nabla}-\nabla\|^2|x_{i,t})}       \nonumber\\
&\leq G+\sigma_\sharp,           \label{ap28}
\end{align}
where $\tilde{\nabla}$ can be any stochastic gradient in Assumption \ref{a3} with the corresponding true gradient $\nabla$, $\sigma_\sharp=\sigma_1$ or $\sigma_2$ depending on $\tilde{\nabla}$, the second inequality has used Jensen's inequality, and the last inequality has employed Assumption \ref{a1}.4 and (\ref{17}) in Assumption \ref{a3}.
\end{proof}

\begin{lemma}\label{l6}
Under Assumptions \ref{a1} and \ref{a2}, there holds
\begin{align}
&\mathbb{E}(\sum_{t=1}^T\|\nu_t-{\bf 1}_N\otimes \bar{\nu}_t\|)=O\Big(\frac{N\sqrt{N}\gamma\xi}{1-\xi}\sum_{t=1}^T\alpha_t\Big),               \label{st1}\\
&\mathbb{E}(\sum_{t=1}^T\alpha_t\|y_t-{\bf 1}_N\otimes \bar{y}_t\|^2)=O\Big(\frac{N^4\gamma^4}{(1-\xi^2)^4}\sum_{t=1}^T\alpha_t^3\Big)           \nonumber\\
&\hspace{0.4cm}+O\Big(\frac{\gamma^2}{(1-\xi^2)^2}V_{T,\alpha_t}^g\Big)+O\Big(\frac{N^2\gamma^2\sigma_2^2}{(1-\xi^2)^2}\sum_{t=1}^T\alpha_t\Big).             \label{st3}
\end{align}
\end{lemma}
\begin{proof}
The proof is similar to that of Lemma \ref{l3} along with (\ref{ap28}), which is thus omitted.
\end{proof}

At this position, the proof of Theorem \ref{t2} can be given by a similar argument as in the proof of Theorem \ref{t1}, together with Lemmas \ref{l4}-\ref{l6}, which is omitted here.

\end{document}